\documentclass[openright,a4paper,american,11pt]{amsart}
\usepackage{amsmath,amsfonts,latexsym, amscd, euscript}
\usepackage{amsthm}
\newtheorem{thm}{Theorem}
\newtheorem{lem}[thm]{Lemma}
\newtheorem{defin}[thm]{Definition}
\newtheorem{rem}[thm]{Remark}
\newtheorem{prop}[thm]{Proposition}

\newtheorem{cor}[thm]{Corollary}
\newtheorem{examp}[thm]{Example}
\newtheorem{problem}{Question}

\usepackage[english]{babel}
\usepackage[latin1]{inputenc}
\usepackage{multicol}
\usepackage{enumerate}
\usepackage{color}
\numberwithin{equation}{section}
\numberwithin{thm}{section}

\newcommand{\CC}{{\mathbb C}}
\newcommand{\NN}{{\mathbb N}}
\newcommand{\KK}{{\mathbb K}}

\newcommand{\EE}{{\EuScript E}}

\newcommand{\RR}{{\mathbb R}}
\newcommand{\ZZ}{{\mathbb Z}}
\newcommand{\QQ}{{\mathbb Q}}

\newcommand{\MM}{{\mathfrak M}}
\newcommand{\mm}{{\mathfrak m}}

\newcommand{\IIP}{{\mathcal I}} 
\newcommand{\IIS}{{\mathfrak I}}

\newcommand{\cerosP}{{\large \textsf{V}}}
\newcommand{\cerosS}{{\large \textsc{V}}}

\newcommand{\SSS}{{\sl S}_\omega}
\newcommand{\ordser}[1]{ val}
\newcommand{\inser}[1]{ in}
\newcommand{\inpol}[2]{{ \textsc{In}}_{#2}}
\newcommand{\inPol}[2]{{{ \textsf{In}}_{#1,#2}}}
\newcommand{\ordpol}[2]{{\textsc{val}}_{#2}}
\newcommand{\ordPol}[2]{{{ \textsf{val}}_{#1,#2}}}
\newcommand{\idinpol}[2]{{{\mathcal I}\textsc{n}}_{#2}}
\newcommand{\idinPol}[2]{{{\mathcal I}{\textsf{n}}}_{#1,#2}}
\newcommand{\idmax}[1]{{\mathcal J}_{#1}}
\newcommand{\ceros}[1]{{\bf V}}
\newcommand{\toro}[1]{{\bf T}_{#1}}
\newcommand{\den}{{\bf d}}
\newcommand{\seriedef}{{\bf ser}}
\newcommand{\secuencia}{{\bf seq}}

\begin{document}
\markboth{F. Aroca, G. Ilardi and L. López de Medrano} {Puiseux
power series solutions for systems of equations}
\title{PUISEUX POWER SERIES SOLUTIONS FOR SYSTEMS OF EQUATIONS}
\author{Fuensanta Aroca, Giovanna Ilardi and Luc\'ia L\'opez de
Medrano}

\address{Instituto de Matem\'aticas, Unidad Cuernavaca
 Universidad Nacional Aut\'onoma de M\'exico,
A.P. 273-3 Admon. 3,
Cuernavaca, Morelos, 62251 M\'exico}
\address{Dipartimento Matematica Ed Applicazioni ``R. Caccioppoli''
Universit\`{a} Degli Studi Di Napoli ``Federico II'' Via Cintia -
Complesso Universitario Di Monte S. Angelo 80126 - Napoli - Italia}

%

\maketitle
\begin{abstract}
    We give an algorithm to compute term by term multivariate Puiseux series
    expansions of series arising as local parametrizations of zeroes of systems of
    algebraic equations at singular points. The algorithm is an extension of Newton's method
    for plane algebraic curves replacing the Newton polygon by the tropical variety of the ideal generated by the
    system.  As a corollary we deduce a property of
    tropical varieties of  quasi-ordinary singularities.
\end{abstract}


\section*{Introduction}

  Isaac Newton described an algorithm to compute term
by term the series arising as $y$-roots of algebraic
equations $f(x,y)=0$ \cite["Methodus fluxionum et serierum infinitorum"
]{Newton:1670}. The main tool used in the algorithm is a geometrical
object called the Newton polygon. The roots found belong to a field
of power series called Puiseux series \cite{Puiseux:1850}.

The extension of Newton-Puiseux's algorithm for equations of the
form $f(x_1,\ldots ,x_N,y)=0$ is due to J. McDonald
\cite{JMcDonald:1995}. As can be expected,  the Newton polygon is
extended by the Newton polyhedron.

An extension for systems of equations of the form $\{
f_1(x,y_1,\ldots ,y_M)= {\cdots} = {f_r(x,y_1,\ldots ,y_M)=0} \}$ is
described in \cite{Maurer:1980} using tropism and in
\cite{JensenMarkwig:2008} using tropical geometry.

J. McDonald gives an extension to systems of equations $$\{
f_1(x_1,\ldots ,x_N,y_1,\ldots ,y_M)= {\cdots} = {f_r(x_1,\ldots
,x_N,y_1,\ldots ,y_M)=0} \}$$ using the Minkowski sum of the Newton
Polyhedra. However, this algorithm works only for ``general"
polynomials \cite{JMcDonald:2002}.

  In this note we extend Newton's method to any dimension and
  codimension. The Newton polyhedron of a polynomial is replaced by its normal
  fan.
  The tropical variety comes in naturally as the intersection
  of normal fans. We prove that, in
    an algebraically closed field of characteristic zero,  the algorithm given always works.

The natural field into which to embed the algebraic closure of
polynomials in one variable is the field of Puiseux series. When it
comes to several variables there is a family of fields to choose
from. Each field is determined by the choice of a vector
$\omega\in\RR^N$ of rationally independent coordinates.
The need to choose $\omega$ had already appeared when
working with a hypersurface
\cite{JMcDonald:1995,GonzalezPerez:2000}.
The introduction of the family of fields is done in \cite{ArocaIlardi:2009}.\\

 We start the article recalling the main statements on which the
Newton-Puiseux method for algebraic plane curves relies (Section
\ref{Newton-Puiseux's Method}) and extending these statements to the
general case (Section \ref{The general statement}).

In the complex case, a series of positive order, obtained by the
Newton-Puiseux method for algebraic plane curves, represents a local
parametrization of the curve around the origin. In Section \ref{The
local parametrizations defined by the series} we explain how an
$M$-tuple of series arising as a solution to the general
Newton-Puiseux statement also represents a local parametrization
recalling results form \cite{FAroca:2004}.

 In Section \ref{The fields} we recall the definition given in
\cite{ArocaIlardi:2009} of the family of fields of $\omega$-positive
Puiseux series and their natural valuation. Then, in Section
\ref{The extended ideal} we reformulate the question using the
fields introduced and show how it becomes a lot simpler.

Then we work in the ring of polynomials with coefficients
$\omega$-positive Puiseux series (Sections \ref{Weighted orders and
initial parts}, \ref{Initial Ideals} and \ref{A special case of
Kapranow's Theorem}).

The Newton-Puiseux algorithm is based on the fact that the first
term of a $y$-root is the $y$-root of the equation restricted to an
edge of the Newton polygon. The analogous of this fact is expressed
in terms of initial ideals. In Sections \ref{Weighted orders and
initial parts} and \ref{Initial Ideals}, weighted orders and initial
ideals
 are defined. In Section \ref{A special case of Kapranow's Theorem} we
prove that initial parts of zeroes are zeroes of weighted initial
ideals.

Then we consider ideals in the ring of polynomials $$\KK [x^*,y]:=
\KK [x_1,{x_1}^{-1},\ldots,x_N,{x_N}^{-1},y_1,\ldots y_M]$$ (Section
\ref{Polynomial initial ideals}) and characterize initial ideals
with zeroes in a given torus. This is done in terms of the tropical
variety of the ideal (Section \ref{Tropicalization}).

Sections \ref{omega-data} to \ref{The solutions} are devoted to
explaining the algorithm. In the last section we show the
theoretical implications of the extension of Newton-Puiseux
algorithm by giving a property of
the tropical variety associated to a quasi-ordinary singularity.\\

\section{Newton-Puiseux's Method.}\label{Newton-Puiseux's Method}

Given an algebraic plane curve ${\mathcal C}:=\{f(x,y)=0\}$, the
Newton-Puiseux method constructs all the fractional power series
$y(x)$ such that $f(x,y(x))=0$. These series turn out to be Puiseux
series.

 Newton-Puiseux's method is based on two points:

Given a polynomial $f(x,y)\in\KK [x,y]$
\begin{enumerate}
    \item\label{dos}
    $cx^\mu$ is the first term of a Puiseux series  $y(x)=cx^\mu+...$
    with the property $f(x,y(x))=0$ if and only if
    \begin{itemize}
        \item
        $\frac{-1}{\mu}$
        is the slope of some edge $L$ of the Newton polygon of $f$.
        \item
        $cx^\mu$ is a solution of the characteristic equation associated to
        $L$.
    \end{itemize}
    \item\label{tres}
    If we iterate the method: Take $c_i x^{\mu_i}$ to be a solution of the characteristic equation associated to the edge of slope
    $\frac{-1}{\mu_i}$ of
    $f_i := f_{i-1} (x, y+ c_{i-1} x^{\mu_{i-1}})$ with
    $\mu_i>\mu_{i-1}$.
    We do get a Puiseux series $\sum_{i=0}^{\infty} c_i x^{\mu_i}$   with the property $f(x,y(x))=0$.
\end{enumerate}

In this paper we prove the extension of these points: Point
\ref{dos} is extended in
Section \ref{Tropicalization}
 Theorem
\ref{Extension del punto uno} and, then, Point \ref{tres} in
Section \ref{The solutions}
 Theorem \ref{ultimo teorema}.

Point \ref{dos} is necessary to assure that the sequences in Point
\ref{tres} always exist. But Point \ref{dos} does not imply that any
sequence constructed in such a way leads to a solution. Both results
have led to a deep understanding of algebraic plane curves.

\section{The general statement.}\label{The general statement}

 Take an $N$-dimensional algebraic variety $V\subset \KK^{N+M}$.
There is no hope to find $k\in\NN$ and an $M$-tuple of series
$y_1,\ldots ,y_M$ in $\KK [[x_1^k,\ldots ,x_M^k]]$ such that the
substitution  $x_j\mapsto y_j(x_1,\ldots ,x_N)$ makes $f$
  identically zero for all $f$ vanishing on $V$.
(Parametrizations covering a whole neighborhood of a singularity do
not exist in general.)

McDonald's great idea was to look for series with exponents in
cones. Introducing rings of series with exponents in cones served to
prove Newton-Puiseux's statement for the hypersurface case
\cite{JMcDonald:1995} and has been the inspiration of lots of other
results (both in algebraic geometry \cite{SotoVicente:2006,GonzalezPerez:2000} and differential equations
\cite{TAranda:2002,FArocaJCano:2001}).

In order to give a general statement for all dimension and codimension, we need to recall some definitions of convex geometry:

A {\bf convex rational polyhedral cone} is a subset of
$\mathbb{R}^N$ of the form
\begin{displaymath}
\sigma =\{ \lambda_1v_1+\cdots+\lambda_r v_r \mid \lambda_i\in
\mathbb{R}, \lambda_i\geq 0\},
\end{displaymath}
where $v_1,\dots,v_r\in \mathbb{Q}^N$ are vectors.

 A cone is said
to be {\bf strongly convex} if it contains no nontrivial linear
subspaces.

A {\bf fractional power series} $\varphi$ in $N$ variables is expressed as
\[
\varphi=\sum_{\alpha \in {\QQ}^N} c_\alpha x^\alpha, \qquad c_\alpha
\in {\KK}, \quad x^{\alpha}:=x_1^{\alpha_1} \dots x_N^{\alpha_N}.
\]
The {\bf set of exponents} of $\varphi$ is the set
\[
{\mathcal E}(\varphi ):=\{\alpha\in {\QQ }^N\mid c_\alpha\neq 0\}.
\]

A fractional power series $\varphi$ is a {\bf Puiseux series} when
its set of exponents is contained in a lattice. That is, there
exists $K\in\NN$ such that ${\mathcal E}(\varphi )\subset
{\frac{1}{K}\ZZ}^N$.

Let $\sigma\subset\RR^N$ be a strongly convex cone. We say that a
Puiseux series $\varphi$ has {\bf exponents in a translate of
$\sigma$} when there exists $\gamma\in\QQ^N$ such that ${\mathcal E}
(x^\gamma\varphi )\subset \sigma$.

It is easy to see that the set of Puiseux series with exponents in
translates of a strongly convex cone $\sigma$ is a ring. (But, when
$N>1$, it is not a field).

Given a non-zero vector $\omega\in \RR^N$, we say that a cone
$\sigma$ is {\bf $\omega$-positive} when for all $v\in\sigma$ we have
$v\cdot\omega\geq 0$. If $\omega$ has rationally independent
coordinates, an $\omega$-positive rational cone is always strongly
convex.

Denote by $\cerosP (\IIP)$ the set of common zeroes of the ideal
$\IIP$. Extending Newton-Puiseux's statement for an algebraic
variety of any dimension and codimension is equivalent to answering
the following question:

\begin{problem}\label{problema}Given an ideal ${\IIP}\subset \KK [x_1,\ldots ,x_{N+M}]$
such that the projection
\begin{equation}\label{la proyeccion}
\begin{array}{cccc}
    \pi:
        & \cerosP({\IIP })
            &\longrightarrow
                & \KK^N\\
        & (x_1,\ldots ,x_{N+M})
            & \mapsto
                & (x_1,\ldots ,x_N)
\end{array}
\end{equation}
is dominant and of generic finite fiber.

Given $\omega\in\RR^N$ of rationally independent coordinates. Can one always find
an $\omega$-positive rational cone $\sigma$
and an $M$-tuple $\phi_1,\ldots, \phi_M$ of Puiseux series with
exponents in some translate of $\sigma$ such that \[ f(x_1,\ldots
,x_N,\phi_1(x_1,\ldots ,x_N),\ldots, \phi_M(x_1,\ldots ,x_N))=0.
\]
for any $f\in\IIP$?
\end{problem}

If the projection is not {dominant} the problem has no
solution. If the generic fiber is not {finite} an output
will not be a parametrization.

To emphasize the roll of the projection, the indeterminates will be
denoted by $x_1,\ldots ,x_N,y_1,\ldots y_M$. We will work with an
ideal $\IIP\subset \KK [x,y]:= \KK [x_1,\ldots ,x_N,y_1,\ldots
,y_M]$.

With this notation, the set of common zeroes of $\IIP$ is given by
\[
\cerosP(\IIP )= \{ (x,y)\in \KK^{N+M}\mid f(x,y)=0,\forall
f\in\IIP\}.
\]
\begin{defin}
We will say that an ideal $\IIP\subset\KK [x,y]$ is {\bf
N-admisible} when the Projection (\ref{la proyeccion}) is dominant
and of finite generic fiber.

We will say that an algebraic variety $V\subset \KK^{N+M}$ is
N-admissible when its defining ideal is N-admissible.

Given an N-admissible ideal $\IIP\subset\KK [x,y]$, and a vector
$\omega\in\RR^N$ of rationally independent coordinates; an $M$-tuple
$\phi_1,\ldots ,\phi_M$ solving Question \ref{problema} will be
called an {\bf $\omega$-solution for $\IIP$}.
\end{defin}


\section{The local parametrizations defined by the series.}\label{The local parametrizations defined by the series}

Let $({\mathcal C},(0,0))$ be a complex plane algebraic curve
singularity
\[
(0,0)\in{\mathcal C}:=\{(x,y)\in \CC^2\mid f(x,y)=0\}
\]
where $f$ is a polynomial with complex coefficients.

 Each output of the
Newton-Puiseux method  $y(x)=c_0x^{\mu_0}+..$ with $\mu_0>0$ is a
convergent series in a neighborhood of $0$. This series corresponds
to a multi-valued mapping defined in a neighborhood of the origin
$0\in U\subset\CC$
\[
\begin{array}{cccc}
    \varphi:
        & U
            & \longrightarrow
                & {\mathcal C}\\
        & x
            &\mapsto
                & (x,y(x))
\end{array}
\]
that is compatible with the
 projection
\[
\begin{array}{cccc}
    \pi:
        &{\mathcal C}
            &\longrightarrow
                &\CC\\
        &(x,y)
            &\mapsto
                 &x,
\end{array}
\]
that is, $\pi\circ\varphi$ is the identity on $U$.

When ${\mathcal C}$ is analytically irreducible at $(0,0)$, the
image $\varphi (U)$ is a neighborhood  of the curve at $(0,0)$. The
series $\varphi$ contains all the topological and analytical
information of $({\mathcal C}, (0,0))$ and there are different ways
to recover it (see for example
\cite{Walker:1978,BrieskornKnorrer:1986}).

If $\omega\in{\RR_{>0}}^N$ has rationally independent  positive
coordinates, then the first orthant is $\omega$-positive and we may
suppose that the series of an output of the extended Newton-Puiseux
method has exponents in a cone $\sigma$ that contains the first
orthant.

Let $\sigma$ be a strongly convex cone that contains the first
orthant. In \cite{FAroca:2004} it is shown that (when it is not
empty) the domain of convergence of a series with exponents in a
strongly convex cone $\sigma$ contains an open set $W$ that has the
origin as accumulation point. Moreover, by the results of \cite[Prop
3.4]{FAroca:2004}, the intersection of a finite number of such
domains is non-empty.

Let $V$ be an N-admissible complex algebraic variety embedded in
$\CC^{N+M}$ and let $\omega\in{\RR_{>0}}^N$ be of rationally
independent coordinates.
 Each $M$-tuple of series $(y_1(\underline{x}),\ldots ,y_M(\underline{x}))$ found solving Question \ref{problema}
     corresponds to a multi-valued function defined on an
open set $W\subset\CC^N$ that has the origin as accumulation point
\[
\begin{array}{cccc}
    \varphi:
        &W
            &\longrightarrow
                &V\\
        &\underline{x}
            &\mapsto
                (\underline{x},y_1(\underline{x}),\ldots

                ,y_M(\underline{x}).
\end{array}
\]
The image $\varphi (W)$ contains an open set (a wedge)
 of $V$.

 When

\begin{equation}\label{el valor es positivo}
 \omega\cdot\alpha > 0\quad\text{for all}\quad \alpha\in\bigcup_{j=1,\ldots
,M}\EE (y_j)
\end{equation}
(when for each $j$, $y_j$ does not have constant term and its set of exponents is contained in an $\omega-$positive cone with apex at the origin) the open set has the origin as accumulation point.

 Since analytic continuation is unique, when the origin is an analytically irreducible singularity,
 this parametrization contains all the topological
and analytic information of the singularity.

\section{The field of $\omega$-positive Puiseux series.}\label{The fields}

In all that follows $\omega$ will be a vector in $\RR^N$ of
rationally independent coordinates. We will work with an
algebraically closed field $\KK$ of characteristic zero.

Given a N-admissible ideal we are looking for solutions in the ring
of Puiseux series with exponents in some translate of an
$\omega$-positive cone $\sigma$. The cone $\sigma$ may be different
for different ideals. It is only natural to work with the infinite
union of all these rings.

We say that a Puiseux series $\varphi$ is {\bf $\omega$-positive}
when there exists $\gamma\in\QQ^N$ and an $\omega$-positive cone
$\sigma$ such that ${\mathcal E} (x^\gamma\varphi )\subset \sigma$.
The set of $\omega$-positive Puiseux series was introduced in
\cite{ArocaIlardi:2009} where it was proved that it is an
algebraically closed field. This field is called the \textbf{field
of $\omega$-positive Puiseux series} and will be denoted by $\SSS$.

The vector $\omega$
induces a total order on $\QQ^N$
\[
\alpha\leq \alpha
'\Longleftrightarrow\omega\cdot\alpha\leq\omega\cdot\alpha '.
\]

This gives a natural way to choose the first term of a series in
$\SSS$. This is the order we will use to compute the $\omega$-solutions ``term by
term''.

More precisely, the {\bf order} of an element
$\phi=\sum_{\alpha}c_\alpha x^\alpha$ in $\SSS$ is
\[
        \ordser{\omega} (\phi) :=\min_{\alpha \in {\mathcal
        E}(f)}
        \omega \cdot \alpha
\]
and its {\bf first term} is
\[
    \inser{\omega} (\phi ):=  c_\alpha
    x^\alpha\qquad\text{where}\qquad \omega \cdot \alpha=\ordser{\omega}
    (\phi).
\]
Set $\ordser{\omega}(0) \colon = \infty$ and $\inser{\omega} (0)=0$.

\begin{rem}\label{propiedades de valser e inser}
    For $\phi ,\phi'\in {\sl S}_\omega$
    \begin{enumerate}
        \item
        $\ordser{\omega} (\phi+\phi')\geq\min \{\ordser{\omega} (\phi), \ordser{\omega}
        (\phi')\}$.\label{Propiedad valuacion 1}
        \item
        $\ordser{\omega} (\phi+\phi')\neq\min \{\ordser{\omega} (\phi), \ordser{\omega}(\phi')\}$ if and only if
        $\ordser{\omega} (\phi)=\ordser{\omega} (\phi')$ and $\inser{\omega}(\phi) +\inser{\omega}(\phi')=0$.\label{segunda propiedad valser}
        \item
        $\ordser{\omega} (\phi\cdot\phi')= \ordser{\omega}(\phi) +\ordser{\omega}(\phi')$. Moreover $\inser{\omega} (\phi\cdot\phi')= \inser{\omega}(\phi)\cdot\inser{\omega}(\phi')$.
        \label{multiplication}
        \item
        $\inser{\omega} (\inser{\omega} (\phi))=\inser{\omega} (\phi) .$
        \item\label{orden y ramificacion}
         $\ordser{\omega}(\phi ({x_1}^r,\ldots ,{x_N}^r))=r\ordser{\omega}(\phi (x_1,\ldots ,x_N))$ for  any
         $r\in\QQ$.
    \end{enumerate}
\end{rem}
A map from a ring into the reals with Properties \ref{Propiedad
valuacion 1} and \ref{multiplication} is called a {\bf valuation}.

The {\bf first $M$-tuple} of an element $\varphi=(\varphi_1,\ldots
,\varphi_M)\in\SSS^M$ is the $M$-tuple of monomials
\[
\inser{\omega}(\varphi)=(\inser{\omega}(\varphi_1),\ldots
,\inser{\omega}(\varphi_M))
\]
and the {\bf order} of $\varphi$ is the $M$-tuple of orders
\[
\ordser{\omega}(\varphi)=(\ordser{\omega}(\varphi_1),\ldots ,\ordser{\omega}(\varphi_M)).
\]

\begin{rem}
With the language introduced,
Equation (\ref{el valor es positivo})
is equivalent to $\ordser{omega} (y)\in {\RR_{>0}}^M$.

\end{rem}

\section{The extended ideal.}\label{The extended ideal}



Given an ideal $\IIP\subset \KK [x,y]$, let $\IIP^*\subset \KK
[x^*,y]$ be the extension of $\IIP$ to $\KK [x^*,y]$ via the natural
inclusion.

We have
\[
 {\bf V} ( \IIP^*\cap \KK [x,y])= \overline{ {\bf V}(\IIP) \setminus
 \{ x_1\cdots x_N=0\}}.
\]

In regard to our question, it is then equivalent to work with ideals
in $\KK [x,y]$ or in $\KK[x^*,y]$. For technical reasons we will
start with ideals in $\KK[x^*,y]$.

\begin{defin}
And ideal $\IIP\subset\KK [x^*,y]$ is said to be {\bf N-admissible}
if the ideal $\IIP\cap\KK [x,y]\subset\KK [x,y]$ is N-admissible.
\end{defin}

Given an ideal $\IIP\subset \KK[x^*,y]$, let ${\IIP}^{\rm
e}\subset\SSS [y]$ be the extension of $\IIP$
via the natural
inclusion
\[
\KK [x^*, y]=\KK[x^*][y]\hookrightarrow {\sl S}_{\omega} [y].
\]
When $\IIP$ is an N-admissible ideal,  $\cerosS ({\IIP}^{\rm e})$ is
a discrete subset of ${\SSS}^M$. By definition, $\phi\in\cerosS
({\IIP}^{\rm e})$, if and only if $\phi$ is an $\omega$-solution for
$\IIP$.

Question \ref{problema} may be reformulated as follows:\\

\begin{problem}{\bf Reformulation of Question \ref{problema}} Given an
N-admissible ideal $\IIP\subset\KK [x^*,y]$, and a vector
$\omega\in\RR^N$ of rationally independent coordinates. Find the
(discrete) set of zeroes of in ${\SSS}^M$ of the extended ideal
${\IIP}^{\rm e}\subset \SSS [y]$ .
\end{problem}

A polynomial $f\in\KK [x^*,y]$ may be considered a polynomial in
$N+M$ variables with coefficients in $\KK$, or a polynomial in $M$
variables with coefficients in $\KK [x^*]\subset\SSS$. To cope with
this fact we will use a slightly different
notation:
\begin{itemize}
    \item[*]
    $\ordser\omega$ and $\inser\omega$ refer to the field $\SSS$.
    (Section \ref{The fields}.)
    \item[*]
    $\ordpol\omega\eta$, $\inpol\omega\eta$ and $\idinpol\omega\eta$ refer to the
    ring  $\SSS [y]$. (Sections \ref{Weighted orders and initial parts} and
    \ref{Initial Ideals}.)
    \item[*]
     $\ordPol\omega\eta$, $\inPol\omega\eta$ and $\idinPol\omega\eta$ refer to the ring  $\KK [x^*,
     y]$. (Section \ref{Polynomial initial ideals}.)
\end{itemize}
Given an ideal $\IIP\subset\KK [x,y]$ the notation $\cerosP (\IIP)$
will stand for the set of common zeroes of $\IIP$ in $\KK^{N+M}$.
Given an ideal $\IIS\subset\SSS [y]$ the set of common zeroes of
$\IIS$ in ${\SSS}^{M}$ will be denoted by $\cerosS (\IIS)$.

\section{Weighted orders and initial parts in $\SSS [y]$.}\label{Weighted orders and initial parts}

The classical definition of weighted order and initial part
considers as weights only vectors in $\RR^M$. For technical reasons we need to
extend the classical definition to weights in
$\left(\RR\cup\{\infty\}\right)^M$.

A polynomial in $M$ variables with coefficients in $\SSS$ is written in the form
\[
f=\sum_{\beta\in E\subset ({\ZZ_{\geq 0}})^M}\phi_\beta y^\beta,\qquad \phi_\beta\in\SSS,\qquad y^\beta := {y_1}^{\beta_1}\cdots {y_M}^{\beta_M}
\]
where $E$ is a finite set.

Set $\infty\cdot a=\infty$ for $a\in\RR^*$ and $\infty\cdot 0=0$. A vector $\eta\in {(\RR\cup\{\infty\})}^M$ induces a (not necessarily total) order on the terms of $f$.

 The {\bf $\eta$-order of $f$ as an element of ${\sl S}_\omega
[y]$} is
\[
\ordpol{\omega}{\eta}(f):=\min_{\phi_\beta\neq
0}\left(\ordser{\omega}\phi_\beta +\eta\cdot\beta\right)
\]
and, if $\ordpol{\omega}{\eta}f<\infty$, the {\bf $\eta$-initial
part of $f$ as an element of ${\sl S}_\omega [y]$} is
\[
\inpol{\omega}{\eta}(f):=\sum_{\ordser{\omega}\phi_\beta+\eta\cdot\beta=\ordpol{\omega}{\eta}(f)}(\inser{\omega}\phi_\beta)
y^\beta .
\]

\begin{examp}\label{inpol de un binomio}
    Consider a binomial of the form $y^\beta -\phi$ we have
    \[
    \ordpol{\omega}{\eta} (y^\beta -\phi )=
    \left\{
        \begin{array}{ll}
            \eta\cdot\beta
                &\text{if}\quad \eta\cdot\beta\leq\ordser{\omega}(\phi)\\
            \ordser{\omega}(\phi)
                &\text{if}\quad \ordser{\omega}\phi\leq\eta\cdot\beta
        \end{array}
    \right.
    \]
    and
    \[
    \inpol{\omega}{\eta}(y^\beta -\phi )=
    \left\{
        \begin{array}{ll}
            y^\beta
                &\text{if}\quad \eta\cdot\beta <\ordser{\omega}(\phi)\\
            y^\beta -\inser{\omega}(\phi)
                &\text{if}\quad \eta\cdot\beta =\ordser{\omega}(\phi)\\
            \inser{\omega}(\phi)
                &\text{if}\quad \ordser{\omega}(\phi)<\eta\cdot\beta .
        \end{array}
    \right.
    \]
\end{examp}

\begin{lem}
\label{key}
    If $\varphi\in\SSS^M$ is a zero of $f\in\SSS [y]$, then $\inser{\omega} (\varphi)$ is a
    zero of $\inpol{\omega}{\ordser{\omega}\varphi}(f)$.
\end{lem}
\begin{proof}
Set $\eta :=\ordser{\omega}(\varphi)$. For $\phi\in\SSS$ and $\beta\in {\ZZ_{\geq 0}}^M$ the following equality holds:
\begin{equation}\label{relacion entre ordser y ordpol}
    \ordser{\omega} \left(\phi\varphi^\beta\right) \stackrel{\ref{propiedades de valser e inser},\, \ref{multiplication}}{=} \ordser{\omega}(\phi) +\eta\cdot\beta = \ordpol{\omega}{\eta}\phi y^\beta.
\end{equation}
Suppose that $\varphi\in\SSS^M$ is a zero of $f=\sum_{\beta}\phi_\beta y^\beta$, we have
\[
\begin{array}{lcl}
    \sum_{\beta}\phi_\beta \varphi^\beta=0
        & \stackrel{(\ref{relacion entre ordser y ordpol})+
        \ref{propiedades de valser e inser},\,\ref{segunda propiedad valser} }{\Longrightarrow}
            &\sum_{\ordser{\omega}\left(\phi_\beta\varphi^\beta\right) =\ordpol{\omega}{\eta} (f)}
            \inser{\omega}\left(\phi_\beta\varphi^\beta\right)=0\\

        &\stackrel{\ref{propiedades de valser e inser},\, \ref{multiplication}}{\Longrightarrow}
            & \sum_{\ordser{\omega}\left(\phi_\beta\right)+\eta\cdot\beta =\ordpol{\omega}{\eta} (f)}
            \inser{\omega}\phi_\beta {\left(\inser{\omega}\varphi\right)}^\beta=0\\
        &\stackrel{\text{By definition}}{\Longrightarrow}
            & \inpol{\omega}{\eta}(f)\left(\inser{\omega}\varphi\right)=0.
\end{array}
\]
\end{proof}

For any $f\in\SSS [y]$ and $\eta\in {(\RR\cup\{\infty\})}^M$ we have $\inpol{\omega}{\eta}(f)\in \KK ({x^{\frac{1}{K}}})[y]$.

An element of the form $cx^\alpha$ with $c\in\KK$ will be called a
monomial.

\begin{lem}\label{Sistema de coeficiente} Given $f\in\SSS [y]$, let $\mm(x)\in {\KK
({x^{\frac{1}{K}}})}^M$ be an $M$-tuple of monomials. Set $\eta:=
\ordser{\omega} \mm$. We have $\inpol{\omega}{\eta}(f)\in \KK
({x^{\frac{1}{K}}})[y]$ and
\[
\inpol{\omega}{\eta}(f(x,\mm(x)))=0\Longleftrightarrow
\inpol{\omega}{\eta}(f(\underline{1},\mm(\underline{1})))=0.
\]
    An $M$-tuple of monomials $\mm\in {\KK ({x^{\frac{1}{K}}})}^M$ with $\ordser{\omega} \mm=\eta$ is
    a zero of $\inpol{\omega}{\eta}(f)$ as an element of $\KK ({x^{\frac{1}{K}}})[y]$ if and only if $\mm(\underline{1})$
    is a zero of $\inpol{\omega}{\eta}(f(\underline{1},y))$.
\end{lem}

\begin{proof}
    If $\ordser{\omega} (\mm)=\eta$ then $\ordser{omega}x^\alpha m^\beta=\omega\cdot\alpha+\eta\cdot\beta$.
    Since $\omega$ has rationally independent coordinates, $x^\alpha \mm^\beta = a x^\gamma$ where
    $a={\mm(\underline{1})}^\beta\in\KK$ and $\gamma$ is the unique vector in $\QQ^N$ such that $w\cdot\gamma=\omega\cdot\alpha+\eta\cdot\beta$.

    Now write
    \[
    \inpol{\omega}{\eta}(f)=\sum_{\omega\cdot\alpha+\eta\cdot\beta=
    \ordpol{\omega}{\eta}(f)} a_{\alpha ,\beta} x^\alpha y^\beta
    \]
    we have $\sum_{\omega\cdot\alpha+\eta\cdot\beta=\ordpol{\omega}{\eta}(f)} a_{\alpha ,\beta} x^\alpha \mm^\beta=0$ if and only if
    \[
    \sum_{\omega\cdot\alpha+\eta\cdot\beta=\ordpol{\omega}{\eta}(f)} a_{\alpha ,\beta} \frac{x^\alpha \mm^\beta}{x^\gamma}=0\Leftrightarrow
    \sum_{\omega\cdot\alpha+\eta\cdot\beta=\ordpol{\omega}{\eta}(f)} a_{\alpha ,\beta} {\mm(\underline{1})}^\beta=0.
    \]
\end{proof}

\section{Initial Ideals in $\SSS [y]$.}\label{Initial Ideals}

For an $M$-tuple $\eta\in {(\RR\cup\{\infty\})}^M$ we will denote by
$\Lambda (\eta )$ the set of subindexes
\[
\Lambda (\eta ):=\{ i\in\{1,\ldots ,M\}\mid \eta_i\neq\infty\}.
\]

\begin{rem}
    $\ordpol{\omega}{\eta}(f)=\infty$ if and only if $f$ is in the ideal generated by $\{ y_i\mid i\in {\Lambda (\eta )}^{\rm C}\}$.
\end{rem}

Let $\IIS$  be an ideal of ${\sl S}_\omega [y]$ and $\eta\in
{(\RR\cup\infty )}^M$. The {\bf $\eta$-initial part of $\IIS$} is
the ideal of ${\sl S}_\omega [y]$ generated by the $\eta$-initial
parts of its elements:
\[
\idinpol{\omega}{\eta} \IIS=\left< \{ \inpol{\omega}{\eta}f\mid f\in
\IIS\}\cup \{y_i\}_{i\in \Lambda(\eta)^{\rm C}}\right> .
\]

Let ${\mathcal A}$ and ${\mathcal B}$ be ideals. We have
\begin{equation}\label{inicial de la interseccion menor que
interseccion de iniciales}
    \idinpol{\omega}{\eta}\left( {\mathcal A}\cap
    {\mathcal B}\right)\subset \idinpol{\omega}{\eta} {\mathcal A}\cap
    \idinpol{\omega}{\eta}{\mathcal B}
\end{equation}
and
\begin{equation}\label{Parte inicial respeta la inclusion}
    {\mathcal A}\subset {\mathcal B}\Longrightarrow
    \idinpol{\omega}{\eta} {\mathcal A}\subset\idinpol{\omega}{\eta}
    {\mathcal B}.
\end{equation}

 Since ${\mathcal A}\cdot{\mathcal B}\subset {\mathcal A}\cap
{\mathcal B}$ then
\begin{equation}\label{inicial del producto menor que inicial de
interseccion}
    \idinpol{\omega}{\eta}\left( {\mathcal A}\cdot{\mathcal B}\right) \subset \idinpol{\omega}{\eta}\left( {\mathcal A}\cap
    {\mathcal B}\right)
\end{equation}
and, since $\inpol{\omega}{\eta} (a\cdot b)=\inpol{\omega}{\eta}
a\cdot \inpol{\omega}{\eta} b$ then
\begin{equation}\label{producto de iniciales menor que inicial del
producto}
    \idinpol{\omega}{\eta}{\mathcal A}\cdot \idinpol{\omega}{\eta}{\mathcal
    B}\subset \idinpol{\omega}{\eta}\left( {\mathcal A}\cdot{\mathcal
    B}\right).
\end{equation}

Let $A$ be an arbitrary set.
For an M-tuple $y\in A^M$ and a subset $\Lambda \subset  \{ 1,\ldots ,M\}$ we will use the following notation:
\begin{equation}\label{Quedarse solo con unas coordenadas}
y_\Lambda := (y_i)_{i\in\Lambda}.
\end{equation}
Given two subsets $B\subset A$ and $C\subset A$ the set $B^{\Lambda}\times C^{\Lambda^{\rm C}}$  is defined to be:
\[
B^{\Lambda}\times C^{\Lambda^{\rm C}}:=\{ y\in A^M\mid y_\Lambda\in B^{\#\Lambda}\,\text{and}\, y_{\Lambda^{\rm C}}\in C^{\#\Lambda^{\rm C}}\}.
\]

We will use the notation $\toro{\eta}$ for the $\# \Lambda (\eta )$-dimensional torus
\[
\toro{\eta} := {\left(\SSS^*\right)}^{\Lambda (\eta )}\times {\{
0\}}^{{\Lambda (\eta )}^{\rm C}}.
\]
\begin{rem}\label{Contenidos en cierre de toro}
    $\cerosS \left( \idinpol{\omega}{\eta}\IIS\right)\subset \overline{\toro{\eta}}$.
\end{rem}

\begin{examp}
    For a point $\varphi=(\varphi_1,\ldots ,\varphi_M)\in {{\sl S}_\omega}^M$ denote by $\idmax{\varphi}$ be the maximal ideal
    \[
    \idmax{\varphi}=\left< y_1-\varphi_1,\ldots
    ,y_M-\varphi_M\right>\subset {\sl S}_\omega [y].
    \]
    Given $\eta\in {(\RR\cup\{\infty\})}^M$ we have
    \[
    \left\{
        \begin{array}{ll}
            \idinpol{\omega}{\eta}\idmax{\varphi}={\sl S}_\omega [y]
                &\text{if}\quad \ordser{\omega}(\varphi_i)<\eta_i\quad\text{for some}\quad i\in\{1,\ldots ,M\}\\
            y_i\in\idinpol{\omega}{\eta}\idmax{\varphi}
                &\text{if}\quad \ordser{\omega}(\varphi_i)>\eta_i\\
            \idinpol{\omega}{\eta}\idmax{\varphi}=\idmax{\inser{\omega}(\varphi)}
                &\text{if}\quad \ordser{\omega}(\varphi)=\eta .
        \end{array}
    \right.
    \]
    The first two points and the inclusion $\idinpol{\omega}{\eta}\idmax{\varphi}\supset\idmax{\inser{\omega}(\varphi)}$
        in the third are direct consequence of Example \ref{inpol de un binomio}. The inclusion
    $\idinpol{\omega}{\eta}\idmax{\varphi}\subset\idmax{\inser{\omega}(\varphi)}$ in the third point is equivalent to
    $\inser{\omega}(\varphi)\in\cerosS\left( \idinpol{\omega}{\eta}\idmax{\varphi}\right)$ which follows from Lemma \ref{key}.
    And then
    \begin{equation}\label{ceros de la parte inicial de J}
    \toro{\eta}\cap \cerosS\left( \idinpol{\omega}{\eta}\idmax{\varphi}\right)=\left\{
    \begin{array}{ccc}
        \emptyset
            & \text{if}
                & \ordser{\omega}(\varphi)\neq\eta\\
        \inser{\omega}(\varphi)
            & \text{if}
                & \ordser{\omega}(\varphi)= \eta.\\
    \end{array}
    \right.
    \end{equation}
\end{examp}

\section{Zeroes of the initial ideal in $\SSS [y]$.} \label{A special case of Kapranow's Theorem}

Now we are ready to characterize the first terms of the zeroes of
the ideal $\IIS\subset {\sl
    S}_\omega [y]$.
The following is the key proposition to extend Point \ref{dos}
 of Newton-Puiseux's
    method.

\begin{prop}\label{Kapranow finito}
     Let $\IIS\subset {\sl
    S}_\omega [y]$ be an ideal with a finite number of zeroes and
    let
    $\eta$ be an $M$-tuple in ${(\RR\cup\{\infty\})}^M$.

    An element $\phi\in \toro{\eta}$ is a zero of the ideal
    $\idinpol{\omega}{\eta}\IIS$ if and only if $\ordser{\omega}(\phi) =
    \eta$ and there exists
    $\varphi\in\cerosS (\IIS )$ such that $\inser{\omega}(\varphi) =\phi$.
\end{prop}

\begin{proof}
    Given $\varphi=(\varphi_1,\ldots ,\varphi_M)\in {{\sl S}_\omega}^M $ consider the ideal
    \[
    \idmax{\varphi}=\left< y_1-\varphi_1,\ldots
    ,y_M-\varphi_M\right>\subset {\sl S}_\omega [y].
    \]

    Set $H:=\cerosS (\IIS )$. By hypothesis $H$ is  a finite subset of ${{\sl S}_\omega}^M$.
    By the Nullstellensatz
    there exists $k\in\NN$ such that
    \[
    {\left( \bigcap_{\varphi\in H} {\mathcal
    J}_{\varphi}\right)}^k\subset \IIS\subset
    \bigcap_{\varphi\in H} {\mathcal J}_{\varphi}.
    \]
    By (\ref{Parte inicial respeta la inclusion}) and (\ref{producto de iniciales menor que inicial del producto})
    we have
    \begin{equation}\label{Meto la parte inicial del ideal entre
    dos}
    {\left( \idinpol{\omega}{\eta}\bigcap_{\varphi\in H} {\mathcal
    J}_{\varphi}\right)}^k\subset \idinpol{\omega}{\eta}\IIS\subset
    \idinpol{\omega}{\eta}\bigcap_{\varphi\in H} {\mathcal J}_{\varphi}.
    \end{equation}
    On the other hand
    \begin{equation}\label{Entre la interseccion y el producto}
        \prod_{\varphi\in H} \idinpol{\omega}{\eta}\idmax{\varphi}
        \stackrel{(\ref{producto de iniciales menor que inicial del
    producto})+(\ref{inicial del producto menor que inicial de
    interseccion})}{\subset}
    \idinpol{\omega}{\eta}\bigcap_{\varphi\in
    H}\idmax{\varphi} \stackrel{(\ref{inicial de la interseccion menor
    que interseccion de iniciales})}{\subset}
    \bigcap_{\varphi\in
    H}\idinpol{\omega}{\eta}\idmax{\varphi}.
    \end{equation}
The zeroes of the right-hand and left-hand side of Equation (\ref{Entre la interseccion y el producto}) coincide. Therefore
    \begin{equation}
        \cerosS\left( \idinpol{\omega}{\eta}\bigcap_{\varphi\in H}\idmax{\varphi}\right)
        \stackrel{(\ref{Entre la interseccion y el producto})}{=}
        \cerosS\left( \bigcap_{\varphi\in
        H}\idinpol{\omega}{\eta}\idmax{\varphi}\right)
        =
        \bigcup_{\varphi\in
        H}\cerosS\left( \idinpol{\omega}{\eta}\idmax{\varphi}\right)
    \end{equation}
    and then, by (\ref{ceros de la parte inicial de J}),
    \begin{equation}\label{ceros de No se que poner}
        \toro{\eta}\cap
        \cerosS\left(\idinpol{\omega}{\eta}\bigcap_{\varphi\in H}
        \idmax{\varphi}\right)=\{\inser{\omega}\varphi\mid\varphi\in H,\ordser{\omega}(\varphi) =
        \eta\}.
    \end{equation}
    The conclusion follows directly from (\ref{Meto la parte inicial del ideal entre
    dos}) and (\ref{ceros de No se que poner}).
\end{proof}

\begin{cor}\label{Los ceros del inicial son monomios de orden eta}

    Let $\IIS\subset {\sl
    S}_\omega [y]$ be an ideal with a finite number of zeroes and
    let
    $\eta$ be an $M$-tuple in ${(\RR\cup\{\infty\})}^M$.

    The zeroes of the ideal
    $\idinpol{\omega}{\eta}\IIS$ in $\toro{\eta}$
    are $M$-tuples of monomials of order $\eta$.
\end{cor}

\section{Initial ideals in $\KK [x^*,y]$.}\label{Polynomial initial ideals}

A polynomial $f\in\KK [x^*, y]=\KK[x^*][y]$ is an expression of the form:
\[
    \sum_{(\alpha ,\beta )\in (\ZZ^N\times {\ZZ_{\geq 0}})^M} a_{(\alpha ,\beta )} x^\alpha y^\beta\qquad a_{(\alpha ,\beta )}\in\KK.
\]

 The $\eta$-order of $f\in\KK [x^*][y]$ as an element of $\SSS [y]$
is called the {\bf $(\omega ,\eta)$-order} of $f$. That is
\[
        \ordPol{\omega}{\eta} (f) :=\min_{a_{(\alpha ,\beta ) }\neq 0}
        \omega\cdot\alpha +\eta\cdot\beta.
\]
And the $\eta$-initial part of $f$ as an element of $\SSS [y]$ is
called the {\bf $(\omega ,\eta)$-initial part} $f$. That is: if
$\ordPol{\omega}{\eta}f<\infty$, then
\[
    \inPol{\omega}{\eta} (f) := \sum_{\omega\cdot\alpha +\eta\cdot\beta=\ordPol{\omega}{\eta} (f)} a_{(\alpha,\beta )} x^\alpha y^\beta
\]
 and,
 if $\ordPol{\omega}{\eta}(f)=\infty$, $\inPol{\omega}{\eta} (f)=0$.

Given an ideal ${\IIP}\subset\KK [x^*][y]$ the {\bf $(\omega ,\eta
)$-initial ideal of ${\IIP}$} is the ideal
\[
\idinPol{\omega}{\eta}{\IIP}:= \left<
\{\inPol{\omega}{\eta}(f)\mid f\in {\IIP}\}\cup \{ y_i\}_{i\in
{\Lambda (\eta )}^{\rm C}}\right>\subset \KK [x^*][y].
\]

Given an ideal ${\IIP}\subset\KK [x^*][y]$ let ${\mathcal
I}^{\rm e}$ denote the extension of ${\mathcal I}$ to $\SSS [y]$.
\begin{prop}\label{las extensiones y sin extender}
    Given $\eta\in {(\RR\cup\{\infty\})}^M$ and an ideal ${\IIP}\subset\KK [x^*,y]$ we have that
    \[
    {\left(\idinPol{\omega}{\eta} {\IIP}\right)}^{\rm e}=
    \idinpol{\omega}{\eta} {\IIP}^{\rm e}.
    \]
\end{prop}
\begin{proof}
The inclusion ${\left(\idinPol{\omega}{\eta} {\IIP}\right)}^{\rm e}\subset
    \idinpol{\omega}{\eta} {\IIP}^{\rm e}$ is straightforward.

    Now, $h\in \{\inpol{\omega}{\eta}f\mid f\in{\IIP}^{\rm e}\}$ if and
    only if $h= \inpol{\omega}{\eta}(\sum_{i=1}^r g_i P_i)$ where
    $g_i\in\SSS [y]$ and $P_i\in {\IIP}$.

    Let $\Lambda =\{ i\mid
    \ordpol{\omega}{\eta}\left( g_iP_i\right)=\min_{j=1,\ldots
    r}\ordpol{\omega}{\eta}\left( g_jP_j\right)\}$. If $\sum_{i\in\Lambda} \inpol{\omega}{\eta} \left( g_i P_i\right) = 0$ then $h= \inpol{\omega}{\eta}(\sum_{i=1}^r g_i')
    P_i $ where $g_i'= g_i-\inpol{\omega}{\eta}(g_i)$ for $i\in\Lambda$
    and $g_i'=g_i$ otherwise.

    Then we can suppose that $\sum_{i\in\Lambda} \inpol{\omega}{\eta} \left( g_i P_i\right)\neq 0$. Then
     $h= \sum_{i\in\Lambda} \inpol{\omega}{\eta}  \left(g_i
    P_i\right) =\sum_{i\in\Lambda} \inpol{\omega}{\eta} (g_i )\inpol{\omega}{\eta}
    (P_i)$ is an element of ${\left(\idinPol{\omega}{\eta} {\IIP}\right)}^{\rm
    e}$.
\end{proof}

We will be using the following technical result:
\begin{lem}\label{inicial del inicial}
Given $\eta\in {(\RR\cup\{\infty\})}^M$ and an ideal ${\mathcal
    I}\subset\KK [x^*][y]$ we have that
\[
\idinPol{\omega}{\eta}( {\idinPol{\omega}{\eta} {\IIP}})
=\idinPol{\omega}{\eta} {\IIP}.
\]
\end{lem}

\begin{proof}It is enough to see that for any $g\in{\idinPol{\omega}{\eta} {\IIP}}$ there exists $f\in \IIP$ such that
${\idinPol{\omega}{\eta} {g}}={\idinPol{\omega}{\eta} (f)}$:

Given $p=\sum_{i=1}^d a_i x^{\alpha_i}y^{\beta_i}\in\KK [x^*][y]$ and $h\in {\IIP}$, we have
\[
p\idinPol{\omega}{\eta}(h)= \sum_{i=1}^d a_i x^{\alpha_i}y^{\beta_i}
\idinPol{\omega}{\eta}(h)=\sum_{i=1}^d\idinPol{\omega}{\eta} (a_i x^{\alpha_i}y^{\beta_i}h).
\]
Then the product $p\idinPol{\omega}{\eta}(h)$ is a sum of $({\omega},{\eta})$-initial parts of elements of ${\IIP}$.

Therefore, $g\in\idinPol{\omega}{\eta}{\IIP}$ if and only if there exists
$f_1,\ldots ,f_r\in {\IIP}$, such that
$g=\sum_{i\in\{ 1,\ldots , r\}} \idinPol{\omega}{\eta}(f_i)$. The $f_i$'s may be chosen such that
$\sum_{i\in\Lambda}\idinPol{\omega}{\eta}(f_i)\neq 0$ for all non-empty $\Lambda\subset\{ 1,\ldots , r\}$. Let
$m=\min_{i\in\{ 1,\ldots , r\}} \ordPol{\omega}{\eta}(f_i)$. Since $\sum_{\ordPol{\omega}{\eta} (f_i)=m}\idinPol{\omega}{\eta} (f_i)\neq 0$ then
$\idinPol{\omega}{\eta} (g)=\sum_{\ordPol{\omega}{\eta} (f_i)=m} \idinPol{\omega}{\eta} (f_i)$, and then $f:=\sum_{\ordPol{\omega}{\eta} (f_i)=m}f_i$
has the property we were looking for.
\end{proof}


\begin{prop}\label{primera traduccion del teorema}
    Let
     $\omega\in\RR^N$ be of rationally independent coordinates,
     let
    $\eta$ be an $M$-tuple in ${(\RR\cup\{\infty\})}^M$ and let $\IIP\subset \KK
    [x^*,y]$ be an N-admisible ideal.

    An element $\phi\in \toro{\eta}$ is an $\omega$-solution for the ideal
    $\idinPol{\omega}{\eta}\IIP$ if and only if $\ordser{\omega}(\phi) =
    \eta$ and there exists $\varphi\in {\SSS}^M$, an $\omega$-solution for $\IIP$,
    such that $\inser{\omega}(\varphi) =\phi$.
\end{prop}
\begin{proof}
This is a direct consequence of Proposition \ref{las extensiones y sin
extender} and Proposition \ref{Kapranow finito}.
\end{proof}

\section{The tropical variety.}\label{Tropicalization}

The tropical variety of a polynomial $f\in \KK [x^*, y]$ is the
$(N+M-1)$-skeleton of the normal fan of its Newton polyhedron. The
tropical variety of an ideal ${\IIP}\subset \KK [x^*, y]$ is
the intersection of the tropical varieties of the elements of
${\IIP}$.

More precisely, the {\bf tropical variety} of ${\IIP}$ is the set
\[
\tau ({\IIP}):= \{(\omega ,\eta )\in\RR^N\times
{(\RR\cup\{\infty\})}^M\mid \idinPol{\omega}{\eta}{\IIP}\cap
\KK [x^*,y_{\Lambda (\eta )}]\text{ does not have a monomial}\}.
\]

  Tropical varieties have become an important tool for solving problems in algebraic geometry.
See for example
\cite{ItenbergMikhalkin:2007,Gathmann:2006,RichterSturmfels:2005}.
In \cite{Bogart:2007,HeptTheobald:2007} algorithms to compute tropical varieties are
described.

\begin{prop}\label{proposicion Tropicalizacion}
    Let ${\IIP}$ be an ideal of $\KK [x^*,y]$. Given
    $\eta \in {(\RR\cup\{\infty\})}^{M}$ the ideal
    $\idinPol{\omega}{\eta}{\IIP}$ has an $\omega$-solution in $\toro{\eta}$
    if and only if $(\omega ,\eta)$ is in the tropical variety of ${\IIP}$.
\end{prop}

\begin{proof}
    Suppose that $\varphi\in \toro{\eta}$ is an $\omega$-solution of $\idinPol{\omega}{\eta}{\IIP}$
    and that $c x^\alpha y^\beta\in \idinPol{\omega}{\eta}{\IIP}\cap \KK [x^*,y_{\Lambda (\eta )}]$. We have
    $x^\alpha\varphi^\beta=0$ and then, $\varphi_i=0$ for some $i\in\Lambda (\eta )$ which
    gives a contradiction.

    Let $\KK (x)$ denote the field of fractions of $\KK [x]$ and let
    $\widetilde{{\IIP}}$ be the extension of $\idinPol{\omega}{\eta}{\IIP}$ to $\KK (x)[y]$
    via the natural inclusion
    $ \KK [x,y]=\KK [x][y]\subset \KK (x)[y]$.

    Since $\SSS$ contains the algebraic closure of $\KK
    (x)$, the zeroes of $\idinPol{\omega}{\eta}{\IIP}$ are the algebraic zeroes
    of $\tilde{{\IIP}}$.  Suppose that $\idinPol{\omega}{\eta}{\IIP}$ does not have
    zeroes in $\toro{\eta}$ then, by Remark \ref{Contenidos en cierre de toro}
    \[
        \cerosS\left(\idinPol{\omega}{\eta}{\IIP}\right)\subset \overline{\toro{\eta}}\setminus\toro{\eta}.
    \]

    Let $v$ be the only element of ${\{ 1\} }^{\Lambda (\eta )}\times {\{ 0\} }^{{\Lambda (\eta )}^{\rm C}}$.
    The monomial $y^v$
    vanishes in all the algebraic zeroes
    of $\tilde{{\IIP}}$. By the Nullstellensatz, there exists $k\in\NN$ such that $y^{kv}$ belongs to $\tilde{{\IIP}}$.

    Now  $y^{kv}$ belongs to $\tilde{{\IIP}}$ if and only if there exists
    $h_1,\dots ,h_r\in \KK [x]\setminus\{ 0\}$ and $f_1,\dots ,f_r\in
    \idinPol{\omega}{\eta}{\IIP}$ such that
    \[
    y^{kv}=\sum_{i=1}^r \frac{1}{h_i}f_i
    \Rightarrow \left(\prod_{i=1}^r h_i(x)\right) y^{kv} =
    \sum_{i=1}^r \left(\prod_{\begin{array}{c}j=1\\i\neq j\end{array}}^r h_j\right) f_i\in
    \idinPol{\omega}{\eta}{\IIP}.
    \]
    Then, by Lemma \ref{inicial del inicial}, $\inPol{\omega}{\eta} \left(\left(\prod_{i=1}^r h_i(x)\right)
    y^{kv}\right)\in\idinPol{\omega}{\eta}{\IIP}$ and $\inPol{\omega}{\eta} \left(\left(\prod_{i=1}^r h_i(x)\right)
    y^{kv}\right)= \inser{\omega}\left(\prod_{i=1}^r h_i(x)\right)y^{kv}$ is a monomial.
    And the result is proved.
\end{proof}

As a direct consequence of of Propositions \ref{primera traduccion del
teorema} and \ref{proposicion Tropicalizacion} we have the extension Point 1 of Newton-Puiseux's
    method.

\begin{thm}\label{Extension del punto uno}
     Let $\IIP\subset \KK [x^*,y]$ a N-admissible ideal and let
     $\omega\in\RR^N$ be of rationally independent coordinates.

    $\phi= (c_1x^{\alpha^{(1)}},\ldots ,c_M x^{\alpha^{(M)}})$ is the first term of
    an
    $\omega$-solution of $\IIP$ if and only if
    \begin{itemize}
        \item
        $(\omega ,\ordser\omega\phi)$ is in the tropical variety of $\IIP$.

        \item $\phi$ is an $\omega$-solution of the ideal $\idinPol{\omega}{\ordser\omega\phi}
        \IIP$.

    \end{itemize}
\end{thm}

These statements recall Kapranov's theorem. Kapranov's theorem was
proved for hypersurfaces in \cite{EinsiedlerKapranov:2006} and the
first published proof for an arbitrary ideal may be found in
\cite{Draisma:2008}. There are several constructive proofs
\cite{Payne:2009,Katz:2009} in the literature. An other proof of
Proposition \ref{Kapranow finito} could probably be done by using
Proposition \ref{las extensiones y sin extender}, showing that
$(\omega ,\eta)\in {\mathcal T} ({\IIP})$ if and only if $\eta\in
{\mathcal T} ({\IIP}^e)$, and checking each step of one of the
constructive proofs.

\section{$\omega$-set.}\label{omega-data}

At this stage we need to introduce some more notation: Given a
$M\times N$ matrix
\[
\Gamma=\left(\!\!\begin{array}{ccc}
        \tiny\Gamma_{1,1}
            & \ldots
                & \tiny\Gamma_{1,N}\\
        \vdots
            &   & \vdots\\
        \tiny\Gamma_{M,1}
            & \ldots
                & \tiny\Gamma_{M,N}
    \end{array}\!\!\right).
\]
The $i$-th row will be denoted by $\Gamma_{i,*}:=
(\Gamma_{i,1},\ldots ,\Gamma_{i,N})$ and

\[
x^\Gamma :=\left(\!\!\begin{array}{c} x^{\tiny\Gamma_{1,*}}\\ \vdots
\\ x^{\tiny\Gamma_{M,*}}
\end{array}\!\!\right).
\]
In particular, if $I\in {\mathcal M}_{N\times N}$ is the identity,
then $x^{\frac{1}{k}I}=({x_1}^{\frac{1}{k}},\ldots
,{x_N}^{\frac{1}{k}})$.

An $M$-tuple of monomials $\mm\in {\KK ({x^{\frac{1}{K}I}})}^M$ can
be written as an entrywise product
\[
 \mm=x^{\Gamma}c=\left(\begin{array}{c}
        c_1 x^{\Gamma_{1,*}}\\
        \vdots\\
        c_M x^{\Gamma_{M,*}}
    \end{array}\right).
\]

Given an $M$-tuple of monomials $\mm\in {\KK
({x^{\frac{1}{K}I}})}^M$ the {\bf defining data of $\mm$} is the
 $3$-tuple
 \[
 D(\mm)=\{\ordser{\omega}\mm ,\Gamma ,\mm(\underline{1})\}
 \]
 where $\Gamma\in {\mathcal M}_{M\times N}(\QQ\cup\{\infty\})$ is the unique matrix such that $\omega\cdot\Gamma^T=\ordser{\omega}{\bf
 m}$ and $\Gamma_{i,*}=\underline{\infty}$ for all $i\in {\Lambda (\ordser{\omega}{\bf
 m})}^{\rm C}$.

\begin{examp} If $\omega = (1,\sqrt{2})$ and
\[
\mm = \left(\begin{array}{c}
         3{x_1}^{3}\\
         7{x_1}^{2}{x_2}\\
         0
    \end{array}\right)
\]
then
\[
    D(\mm) =\{ (3,2+\sqrt{2},\infty ),\left(\begin{array}{cc}
            3 & 0\\
            2 & 1\\
            \infty & \infty
        \end{array}\right), (3,7,0)\}.
    \]
\end{examp}

\begin{defin}
 An {\bf $\omega$-set} is a $3$-tuple $\{ \eta,\Gamma,c\}$ where
\begin{equation}\label{Donde estan los elementos de un starting data}
 \eta\in (\mathbb{R}\cup\{\infty\})^M,\,\,\Gamma\in
 {\mathcal M}_{M\times N}(\mathbb{Q}\cup \{\infty\}),\, c\in
 \KK^M
\end{equation}
and
\begin{itemize}
    \item
    $\omega\cdot\Gamma^T=\eta$
    \item
    $\Gamma_{i,*}=\underline{\infty}$ for all $i\in {\Lambda (\eta
    )}^{\rm C}$
    \item
    $c\in
    {\KK^*}^{\Lambda (\eta)}\times {\{ 0\}}^{{\Lambda (\eta)}^{\rm
    C}}$.
\end{itemize}
\end{defin}

 Given  an $\omega$-set $D=\{\eta,\Gamma, c\}$ the {\bf M-tuple defined by $D$} is the M-tuple of monomials
\[
\MM_D: = x^{\Gamma}c.
\]
We have
\[
\MM_{\{\eta ,\Gamma
    ,c\}}(x^{rI})=\MM_{\{r\eta ,r\Gamma
    ,c\}}(x).
\]
\begin{rem}
$\mm=\MM_{D(\mm)}$ and $D(\MM_D)=D$.
\end{rem}

\section{Starting $\omega$-set for $\mathcal{I}$.}\label{omega-starting data}

Given an N-admissible ideal ${\IIP}\subset \KK[x^*,y]$.  {\bf
A starting $\omega$-set for ${\IIP}$} is  an $\omega$-set $D=\{
\eta,\Gamma,c\}$ such that
\begin{itemize}
    \item  The vector $(\omega,\eta)$ is in the tropical variety of ${\IIP}$.
    \item $c$ is a zero of the system $\{f(\underline{1},y)=0\mid f\in \text{In}_{\omega,\eta}\IIP\}$.
\end{itemize}

\begin{examp}
    Let ${\IIP}=\left< x_1+y_1-y_2+y_1y_2+y_3,x_2-y_1+y_2+2y_1 y_2, y_3\right>$.
    For $\omega =(1,\sqrt{2})$ there are two possible starting $\omega$-sets
    \[
    D1=\{ (1,1,\infty ),\left(\begin{array}{cc}
            1 & 0\\
            1 & 0\\
            \infty & \infty
        \end{array}\right), (1,1,0)\}
    \]
    and
    \[
    D2=\{ (0,0,\infty ),\left(\begin{array}{cc}
            0 & 0\\
            0 & 0\\
            \infty & \infty
        \end{array}\right), (\frac{1}{3},\frac{1}{5},0)\}.
    \]
and
    \[
    \MM_{D1}(x)=\left(\begin{array}{c}
         x_1\\
         x_1\\
         0
    \end{array}\right),\,
    \MM_{D1} (x^{\frac{1}{3}I})=\left(\begin{array}{c}
         {x_1}^{\frac{1}{3}}\\
         {x_1}^{\frac{1}{3}}\\
         0
    \end{array}\right)\,\text{and}
    \,\MM_{D2}(x)=\left(\begin{array}{c}
         \frac{1}{3}\\
         \frac{1}{5}\\
         0
    \end{array}\right).
    \]
\end{examp}

\begin{prop}\label{Mupla asociada a data}
 The $\omega$-set $D=\{\eta ,\Gamma ,c\}$ is  a starting $\omega$-set for ${\IIP}$ if and only if
  $\MM_D$ is an $\omega$-solution of $\idinPol{\omega}{\eta} {\IIP}$.

  Moreover all the $\omega$-solutions of
 $\idinPol{\omega}{\eta} {\IIP}$ in $\toro{\eta}$ are of the form $\MM_D$ where $D=\{\eta ,\Gamma ,c\}$ is  a starting $\omega$-set for ${\IIP}$.
\end{prop}

\begin{proof}
That $\MM_D$ is an $\omega$-solution of $\idinPol{\omega}{\eta}
{\IIP}$ when $D$ is  a starting $\omega$-set is a direct consequence
of Lemma \ref{Sistema de coeficiente}. The other implication is
consequence of Proposition \ref{proposicion Tropicalizacion} and
Lemma \ref{Sistema de coeficiente}.

The last sentence follows from Corollary \ref{Los ceros del inicial
son monomios de orden eta}.
\end{proof}

\section{The ideal ${\IIP}_D$.}\label{The ideal ID}

Given a matrix $\Gamma\in {\mathcal M}_{M\times
N}(\QQ\cup\{\infty\})$  the minimum common multiple of the
denominators of its entries will be denoted by $\den\Gamma$. That is
\[
\den\Gamma := \min\{ k\in\NN\mid \Gamma\in {\mathcal M}_{N\times
M}(\mathbb{Z}\cup\{\infty\})\}.
\]

Given  an $\omega$-set $D=\{\eta,\Gamma, c\}$, we will denote by
${\IIP}_D$
 the ideal in $\mathbb{K}[x^*,y]$ given by
\[
{\IIP}_D:= \left<\{f(x^{\den\Gamma I},y+\MM_D(x^{\den\Gamma I})\mid
f\in \mathcal{I}\}\right>\subset\KK [x^*,y].
\]

\begin{rem}\label{ceros de I y de ID} A series $\phi\in\SSS^M$ is an $\omega$-solution of ${\IIP}$ if and only if the series
$\tilde{\phi}:= \phi(x^{\den\Gamma I})- \MM_D(x^{\den\Gamma I})$ is
an $\omega$-solution of ${\IIP}_D$.
\end{rem}

\begin{prop}\label{En la tropicalizacion hay uno de pendiente mayor}
Let $D=\{\eta,\Gamma, c\}$ be  a starting $\omega$-set for an ideal
${\mathcal I}$. There exists
$\tilde{\eta}\in(\mathbb{R}\cup\{\infty\})^M$ such that
$(\omega,\tilde{\eta})\in\tau ({\mathcal I}_D)$ and
$\tilde{\eta}_{\Lambda (\tilde{\eta})}>\den\Gamma\eta_{\Lambda
(\tilde{\eta})}$ coordinate-wise.
\end{prop}

\begin{proof}

 By Proposition \ref{Kapranow finito} and Proposition \ref{Mupla asociada a data}, $\MM_D$ is the first term of at least one  $\omega$-solution of
$\mathcal{I}$. Say
\[
\phi=\MM_D+\tilde{\phi}\in\cerosS({\mathcal I}),\quad \tilde{\phi}=
\left(\begin{array}{c}
         \tilde{\phi}_1\\
         \vdots\\
         \tilde{\phi}_M
    \end{array}\right)\in \SSS^M,
    \]
    with
    $\ordser{\omega}(\tilde{\phi}_i)>
    \omega\cdot\Gamma_{i,*}=\eta_i
    $ when $\tilde{\phi}_i\neq 0$.

Set
\[
\tilde{\eta}:=\den\Gamma\ordser{\omega}(\tilde{\phi})\stackrel{\text{Remark}
\ref{propiedades de valser e inser},\ref{orden y
ramificacion}}{=}\ordser{\omega}(\tilde{\phi})(x^{\den\Gamma I})
\]
 then
$\tilde{\eta}_i>\den\Gamma\eta_i$ for all $i\in\Lambda
(\tilde{\eta})$.

By Remark \ref{ceros de I y de ID} $\tilde{\phi}(x^{\den\Gamma I})$
is an $\omega$-solution of ${\mathcal I}_D$. Then, by Theorem \ref{Kapranow
finito}, $\inser{\omega}\tilde{\phi}$ is an $\omega$-solution of
$\idinPol{\omega}{\tilde{\eta}}{\mathcal I}_D$. Finally, by Proposition
\ref{proposicion Tropicalizacion}, $(\omega ,\tilde{\eta})$ is in
the tropical variety of ${\mathcal I}_D$.
\end{proof}

\begin{prop}\label{Hay uno de pendiente mayor}
    Let $D=\{\eta,\Gamma, c\}$ be  a starting $\omega$-set for an ideal ${\mathcal I}$.
    There exists  a starting $\omega$-set $D'=\{\eta',\Gamma',
    c'\}$ for ${\mathcal I}_D$ such that ${\eta'}_{\Lambda (\eta')}>\den\Gamma\eta_{\Lambda (\eta')}$ coordinate-wise.
\end{prop}
\begin{proof}
    By Proposition \ref{En la tropicalizacion hay uno de pendiente mayor}, there exists
    $\eta'\in {(\RR\cup\{\infty\})}^M$ such that ${\eta'}_{\Lambda  (\eta')}>\den\Gamma\eta_{\Lambda (\eta')}$
    coordinate-wise and $(\omega ,\eta')\in\tau ({\mathcal I}_D)$. By Proposition
    \ref{proposicion Tropicalizacion}, the ideal
    $\idinPol{\omega}{\eta'}{\mathcal I}_D$ has an $\omega$-solution $\phi$ in $\toro{\eta'}$. By Proposition
    \ref{Mupla asociada a data}, $\phi=\MM_{D'}$ where
    $D'=\{\eta',\Gamma',c'\}$ is a starting $\omega$-set for ${\mathcal I}_D$.
\end{proof}

\section{$\omega$-sequences.}\label{omega-sequences}

Given an $M$-tuple $\phi\in\SSS^M$ define inductively
$\{\phi^{(i)}\}_{i=0}^\infty$, and $\{ D^{(i)}\}_{i=0}^\infty$ by:
\begin{itemize}
    \item{For $i=0$}
    \begin{itemize}
        \item
        $\phi^{(0)}:=\phi$
        \item
        $D^{(0)}$ is the defining data of
        $\inser{\omega}\phi$.
    \end{itemize}
    \item{For $i>0$:}
    \begin{itemize}
        \item
        $\phi^{(i)}:=\phi^{(i-1)}(x^{\den\Gamma^{(i-1)}
        I})-\MM_{D^{(i-1)}}(x^{\den\Gamma^{(i-1)}I})$
        \item
         $D^{(i)}$ is the defining data of the
    $M$-tuple of monomials $\inser{\omega}\phi^{(i)}$.
    ($D^{(i)}:=D(\inser{\omega}\phi^{(i)})$).
    \end{itemize}
\end{itemize}
The sequence above
\[
\secuencia (\phi):= \{D^{(i)}\}_{i=0}^\infty
\]
will be called {\bf the defining data sequence for $\phi$}.
\begin{rem}\label{Los eta crecen en el data asociado a serie}
For any $\phi\in\SSS^M$. If $\secuencia (\phi):= \{\eta^{(i)},\Gamma^{(i)},c^{(i)}\}_{i=0}^\infty$ then
${\eta^{(i)}}_{\Lambda (\eta^{(i)})}>\den\Gamma^{(i-1)}{\eta^{(i-1)}}_{\Lambda (\eta^{(i)})}$ coordinate-wise.
\end{rem}

Given a sequence $S=\{D^{(i)}\}_{i=0\ldots K}=\{\eta^{(i)},\Gamma^{(i)}, c^{(i)}\}_{i=0\ldots K}$, with
$K\in\ZZ_{\geq 0}\cup\{\infty\}$.
Set
\begin{itemize}
    \item ${\mathcal I}^{(0)}={\mathcal I}$
    \item ${\mathcal I}^{(i)}={\mathcal I}^{(i-1)}_{D^{(i-i)}}$ for $i\in\{1,\ldots ,K\}$
\end{itemize}
$S$ is called an {\bf $\omega$-sequence for ${\mathcal I}$} if and
only if for $i\in\{0,\ldots ,K\}$

\begin{itemize}
        \item
        $D^{(i)}=\{\eta^{(i)},\Gamma^{(i)}, c^{(i)}\}$ is a starting $\omega$-set for ${\mathcal I}^{(i)}$
    \item
        ${\eta^{(i)}}_{\Lambda (\eta^{(i)})}>\den\Gamma^{(i-1)}{\eta^{i-1}}_{\Lambda (\eta^{(i)})}$ coordinate-wise.
\end{itemize}
As a corollary to Proposition \ref{Hay uno de pendiente mayor} we
have:
\begin{cor}
    Let $\{D^{(i)}\}_{i=0\ldots K}$ be an $\omega$-sequence for ${\mathcal I}$. For any $K'\in\{ K+1,\ldots ,\infty\}$ there exists a sequence $\{D^{(i)}\}_{i=K+1\ldots K'}$ such that $\{D^{(i)}\}_{i=0\ldots K'}$ is an $\omega$-sequence for ${\mathcal I}$.
\end{cor}

\begin{prop}\label{ceros dan w-secuencia}
 If $\phi$ is an $\omega$-solution of ${\mathcal I}$ then $\secuencia (\phi )$ is an $\omega$-sequence for ${\mathcal I}$.
 \end{prop}
 \begin{proof}
 This is a direct consequence of Remarks
  \ref{Los eta crecen en el data asociado a serie}  and \ref{ceros de I y de
  ID}.
\end{proof}

\section{The solutions.}\label{The solutions}

Given an $\omega$-sequence $S=\{ D^{(i)}\}_{i=0\ldots K}$ for
${\mathcal I}$, with $D^{(i)}=\{\eta^{(i)},\Gamma^{(i)},c^{(i)}\}$
set
\begin{equation}\label{sucesion de ramificaciones}
r^{(0)}:=1\quad\text{and}\quad r^{(i)}:=\frac{1}{\prod_{j=0}^{i-1}\den\Gamma^{(j)}}\,\text{for}\, i>0.
\end{equation}
 The {\bf series defined by $S$} is the series
\[
\seriedef (S):= \sum_{i=0}^K \MM_{D^{(i)}} (x^{r^{(i)}I}).
\]

The following theorem is the extension of Point \ref{tres} of Newton-Puiseux's
    method:

\begin{thm}\label{ultimo teorema} \label{data da ceros}
If $S=\{D^{(i)}\}_{i=0}^\infty$ is an $\omega$-sequence for
${\mathcal I}$ then $\seriedef (S)$ is an $\omega$-solution of
${\mathcal I}$.
\end{thm}
\begin{proof}
Let $S=\{ D^{(i)}\}_{i=1}^\infty$ be an $\omega$-sequence for ${\mathcal I}$. Where $D^{(i)}=\{\eta^{(i)},\Gamma^{(i)},c^{(i)}\}$. Let $\{ {\mathcal I}^{(i)}\}_{i=0}^\infty$ be defined by ${\mathcal I}^{(0)}:={\mathcal I}$ and ${\mathcal I}^{(i)}:={\mathcal I}^{(i-1)}_{D^{(i-1)}}$. We have that $D^{(i)}$ is a starting $\omega$-set for ${\mathcal I}^{(i)}$.

By Proposition \ref{proposicion Tropicalizacion} for each $i\in\NN$
there exists $\phi^{(i)}\in\SSS^M$ such that
$\ordser{\omega}\phi^{(i)}=\eta^{(i)}$ and $\phi^{(i)}\in\cerosS
\left( {\mathcal I}^{(i)}\right)$.

Set $\{r^{(i)}\}_{i=0}^\infty$ as in (\ref{sucesion de
ramificaciones})  and $\tilde{\phi^{(i)}}:= \seriedef
(\{D^{(i)}\}_{j=0}^{i-1})+\phi^{(i)}(x^{r^{(i)} I})$.

For each $i\in\NN$, by Remark \ref{ceros de I y de ID},
$\tilde{\phi^{(i)}}\in\cerosS \left( {\mathcal
I}\right)\subset\SSS^M$.

Since there are only a finite number of zeroes there exists $K\in\NN$ such that $\tilde{\phi^{(i)}}=\tilde{\phi^{(K)}}$
for all $i>K$. Then
\[
\seriedef (S)=\tilde{\phi^{(K)}}\in\cerosS \left({\mathcal
I}\right)\subset\SSS^M.
\]
\end{proof}

Theorem \ref{ultimo teorema} together with Proposition \ref{ceros
dan w-secuencia} gives:
\begin{cor} {\bf Answer to Question \ref{problema}.}\\
    Let ${\mathcal I}\subset\KK[x^*,y]$ be an N-admissible ideal
    and let $\omega\in\RR^N$ be
    of rationally independent coordinates.

The M-tuple of series defined by an $\omega$-sequence for ${\mathcal I}$ is an element of $\SSS^M$.

An M-tuple of series $\phi\in\SSS^M$ is an $\omega$-solution of ${\mathcal I}$ if and only if $\phi$ is an M-tuple of series defined by an $\omega$-sequence for ${\mathcal I}$.
\end{cor}

\section{The tropical variety of a quasi-ordinary
singularity.}

Let $(V,\underline{0})$ be a singular $N$-dimensional germ of
algebraic variety. $(V,\underline{0})$ is said to be {\bf
quasi-ordinary} when it admits a projection $\pi:
(V,\underline{0})\longrightarrow (\KK^N,\underline{0})$ whose
discriminant is contained in the coordinate hyperplanes. Such a
projection is called a {\bf quasi-ordinary projection}.

Quasi-ordinary singularities admit analytic local parametrizations.
This was shown for hypersurfaces by S. Abhyankar
\cite{Abhyankar:1955} and extended to arbitrary codimension in
\cite{FAroca:2004}. Quasi-ordinary singularities have been the
object of study of many research papers
\cite[...]{ArocaSnoussi:2005,Gau:1988,Tornero:2001,Popescu-Pampu:2003}.

\begin{cor}
Let $V$ be an $N$-dimensional algebraic variety embedded in
$\CC^{N+M}$ with a quasi-ordinary analytically irreducible
singularity at the origin. Let
\[
\begin{array}{cccc}
    \pi:
        & V
            & \longrightarrow
                & \CC^N\\
        & (x_1,\ldots ,x_{N+M})
            & \mapsto
                & (x_1,\ldots ,x_N)
\end{array}
\]
 be a quasi-ordinary projection.

Let ${\mathcal I}\subset\KK [x_1,\ldots ,x_{N+M}]$ be the defining
ideal of $V$. Then for any $\omega\in {\RR_{>0}}^N$ of rationally independent coordinates there exists a
{\bf unique} $e\in {\RR_{>0}}^M$ such that $(\omega ,e)$ is in the tropical variety ${\mathcal
T}({\mathcal I})$.
\end{cor}
\begin{proof}

Since $\pi$ is quasi-ordinary there exists $k$ and an $M$-tuple of
analytic series in $N$ variables $\varphi_1,\ldots ,\varphi_M$ such
that $(t_1^k,\ldots ,t_N^k,\varphi_1(\underline{t}),\ldots
,\varphi_M(\underline{t}))$ are parametric equations of $\cerosP
({\mathcal I})$ about the origin \cite{FAroca:2004}.

For any $\omega\in {\RR_{>0}}^N$ of rationally independent
coordinates, the first orthant is $\omega$-positive and then
$(\varphi_1,\ldots ,\varphi_M)$ is an element of ${\SSS}^M$. Set

$\eta_\omega:=\ordser\omega (\varphi_1,\ldots ,\varphi_M)$. The
$M$-tuple of Puiseux series (with positive exponents)
$y:=(\varphi_1({x_1}^{\frac{1}{k}},\ldots
,{x_N}^{\frac{1}{k}}),\ldots ,\varphi_M({x_1}^{\frac{1}{k}},\ldots
,{x_N}^{\frac{1}{k}}))$ is an $\omega$-solution of $\IIP$ with
$\ordser\omega
 (y) =\frac{1}{k}\eta_\omega$
 By Theorem \ref{Extension del punto uno} $(\omega ,\frac{1}{k}\eta_\omega)$ is
 of $\IIP$.

 Now we will show that $(\omega ,e)\in {\mathcal T} (\IIP)$ implies
 $e=\frac{1}{k}\eta_\omega$.

Set $\Phi$ to be the map:
\[
\begin{array}{cccc}
    \Phi:
        & U
            & \longrightarrow
                & \CC^{N+M}\\
        & (t_1,\ldots ,t_N)
            & \mapsto
                & (t_1^k,\ldots ,t_N^k,\varphi_1(\underline{t}),\ldots
,\varphi_M(\underline{t}))
\end{array}
\]

 Since the
singularity is analytically irreducible this parametrization covers
a neighborhood of the singularity at the origin. Let $U\subset\CC^N$
be the common domain of convergence of $\varphi_1,\ldots
,\varphi_M$. $U$ contains a neighborhood of the origin. We have
 $\Phi (U)= \pi^{-1} (U)\cap V$.

Let $y_1,\ldots ,y_M$ be Puiseux series with exponents in some
$\omega$-positive rational cone $\sigma$ (with $\ordser{\omega}
\underline{y}=e\in{\RR_{> 0}}^M$) such that $f(x_1,\ldots ,x_N,
y_1(\underline{x}),\ldots ,y_M(\underline{x}))=0$ for all $f\in
{\mathcal I}$.

Take $k'$ such that $\psi_j(\underline{t}):= y_j({t_1}^{k'},\ldots
,{t_N}^{k'})$ is a series with integer exponents for all $j=1\ldots
M$. (This may be done since they are Puiseux series).

There exists an open set $W\subset\CC^N$ where all the series
$\psi_j$ converge, that has the origin as accumulation point .

Set $\Psi$ to be the map:
\[
\begin{array}{cccc}
    \Psi :
        & W
            & \longrightarrow
                & \CC^{N+M}\\
        & (t_1,\ldots ,t_N)
            & \mapsto
                & ({t_1}^{k'},\ldots ,{t_N}^{k'},
                \psi_1(\underline{t}),\ldots ,\psi_M(\underline{t})).
\end{array}
\]

Since $U$ is a neighborhood of the origin, $W':= U\cap W\neq
\emptyset$.

Ramifying again, if necessary, we may suppose that $k=k'$.

Take $\underline{t}\in W'$. There exists $\underline{t}'\in U$ such

that $\Psi (\underline{t}) = \Phi (\underline{t}')$. Then
$\pi\circ\Psi (\underline{t}) = \pi\circ\Phi (\underline{t}')$, and,
then $({t_1}^k,\ldots ,{t_N}^k)=({t'_1}^k,\ldots ,{t'_N}^k).$

There exists an $N$-tuple of $k$-roots of the unity $\xi_1,\ldots
\xi_N$ such that $(t_1,\ldots ,t_N)=(\xi_1 t'_1,\ldots ,\xi_N
t'_N)$. By continuity this $N$-tuple is the same for all
$\underline{t}\in W'$.

We have $\Psi (t_1,\ldots ,t_N)=\Phi (\xi_1 t_1,\ldots ,\xi_N t_N)$
on $W'$. Since $W'$ contains an open set, we have the equality of
series $\psi_j (t_1,\ldots ,t_N)=\phi_j (\xi_1 t_1,\ldots ,\xi_N
t_N)$.

Then $(y_1,\ldots ,y_M)= (\varphi_1 (\xi_1
{x_1}^{\frac{1}{k}},\ldots ,\xi_N {x_N}^{\frac{1}{k}}),\ldots
,\varphi_M (\xi_1 {x_1}^{\frac{1}{k}},\ldots ,\xi_N
{x_N}^{\frac{1}{k}}))$ for some $N$-tuple of roots of unity
$(\xi_1,\ldots ,\xi_N)$ and then
\[
\ordser{\omega} (y_1,\ldots ,y_M) =\frac{1}{k} \ordser{\omega}
(\varphi_1,\ldots \varphi_M).
\]

The conclusion follows from Theorem \ref{Extension del punto uno}.

\end{proof}

\section*{Closing remarks.} \label{Closing remarks section}

In the literature there are many results relating the Newton
polyhedron of a hypersurface and invariants of its singularities.
See for example \cite{Kouchirenko:1976}. To extend this type of
theorems to arbitrary codimension the usual approach has been to
work with the Newton polyhedrons of a system of generators (see for
example \cite{Oka:1997}). The results presented here suggest that
using the notion of tropical variety better results may be obtained.

Both Newton-Puiseux's and McDonald's algorithm have been extended
for an ordinary differential equation \cite{Fine:1889} and a partial
differential equation \cite{FArocaJCano:2001,FArocaJCanoFJung:2003}
respectively. The algorithm presented here can definitely  be extended to
 systems of partial differential equations; a first step
in this direction can be found in \cite{FAroca:2009}.

\section*{Acknowledgments}
The first and the third author were partially supported by CONACyT
55084 and UNAM: PAPIIT IN 105806 and  IN 102307.\\

During the preparation of this work the third author was profiting
of the post-doctoral fellowship CONACyT 37035.



\def\cprime{$'$}

\end{document}